\documentclass[12pt,a4paper]{amsart}
\usepackage{amsmath}
\usepackage{amsthm}
\usepackage{amssymb}
\usepackage{drpack}
\usepackage[english]{babel}
\usepackage{verbatim}
\usepackage[shortlabels]{enumitem}

\newtheoremstyle{mio}%
{}{} % spazio sopra e sotto%
{\itshape}{} % corpo del testo, indentation
{\bfseries}{.}{ } % titolo del teorema: tipo di testo, divisore, spaziatura
{#1 #2\thmnote{~\mdseries(#3)}} % formattazione nota
\theoremstyle{mio}
\newtheorem{teor}{Theorem}[section]
\newtheorem{cor}[teor]{Corollary}
\newtheorem{prop}[teor]{Proposition}
\newtheorem{lemma}[teor]{Lemma}
\newtheorem{defin}[teor]{Definition}

\newtheoremstyle{definition2}%
{}{} % spazio sopra e sotto%
{}{} % corpo del testo, indentation
{\bfseries}{.}{ } % titolo del teorema: tipo di testo, divisore, spaziatura
{#1 #2\thmnote{\mdseries~ #3}} % formattazione nota
\theoremstyle{definition2}
\newtheorem{ex}[teor]{Example}
\newtheorem{oss}[teor]{Remark}

\newcommand{\recip}{\mathcal{R}}
\newcommand{\local}{\mathcal{O}}
\newcommand{\aff}{\mathbb{A}}
\newcommand{\proj}{\mathbb{P}}
\newcommand{\mm}{\mathfrak{m}}
\newcommand{\pp}{\mathfrak{p}}

\newcommand{\transf}{\mathcal{T}}
\newcommand{\xx}{\mathcal{X}}
\newcommand{\unit}{\mathcal{U}}
\newcommand{\XX}{\mathbf{X}}
\newcommand{\yy}{\mathcal{Y}}
\newcommand{\Vaff}{V}
\newcommand{\Vproj}{\overline{V}}
\newcommand{\divv}{\mathrm{div}}

\title{The reciprocal complement of a surface}
\author{Dario Spirito}
\address{Dipartimento di Scienze Matematiche, Fisiche e Informatiche, Universit\`a di Udine, Udine, Italy}
\email{dario.spirito@uniud.it}
\begin{document}
\begin{abstract}
We study the reciprocal complement $\recip(D)$ of a two-dimensional finitely generated $K$-algebra $D$ by linking it with the properties of a surface with coordinate ring $D$. We give several sufficient criteria to have $\dim\recip(D)=2$, and we use them to show several explicit examples; in particular, we determine the dimension of $\recip(D)$ when $D$ is the quotient of $K[X,Y,Z]$ by an irreducible polynomial of degree $2$. We also study the integral closure of the localizations of $\recip(K[X,Y])$.
\end{abstract}

\maketitle

\section{Introduction}
Given an integral domain $D$, the \emph{reciprocal complement} of $D$ is the domain $\recip(D)$ generated by the inverses of the nonzero elements of $D$. Then, $\recip(D)$ is always an integral domain with the same quotient field as $D$; yet, the properties of $\recip(D)$ are often very distant from those of $D$. For example, when $D=K[X,Y]$ is the polynomial ring in two indeterminates over a field $K$, the reciprocal complement $\recip(D)$ is a two-dimensional local domain that is neither Noetherian nor integrally closed, and where every non-unit is contained in all but finitely many prime ideals \cite{EGL-RD-polinomyal}. Yet, the structure of $\recip(D)$ is very linked to $D$: for example, if $D$ is a finitely generated $K$-algebra then $\dim\recip(D)\leq\dim(D)$ \cite[Theorem 3.2]{guerrieri-recip}. The construction $\recip(D)$ was introduced in \cite{epstein-egypt-23b} during the study of \emph{Egyptian domains}, i.e., domains $D$ where every $d\in D$ can be written as a sum of reciprocals of elements of $D$; a domain is Egyptian if and only if $\recip(D)$ is the quotient field of $D$. The study of $\recip(D)$ is often very hard; see \cite{EGL-RD-polinomyal,guerrieri-recip,elliott-epstein-factroids,recip-compl-survey} for some results about this construction.

In an earlier paper \cite{reciprocal-curves}, a geometrical approach to $\recip(D)$ was pursued when $D$ is a one-dimensional $K$-algebra and $K$ is an algebraically closed field, by connecting $D$ with the affine curves whose coordinate ring is isomorphic to $D$. The main result was a complete characterization of the dimension of $\recip(D)$: if $\xx$ is a curve with $K[\xx]\simeq D$ and the points at infinity are regular for the projective closure $\overline{\xx}$, then $\dim\recip(D)=1$ if $\overline{\xx}\setminus\xx$ is a single point, while $\dim\recip(D)=0$ (that is, $\recip(D)$ is the quotient field of $D$) if $\overline{\xx}\setminus\xx$ contains more than one point.

In the present paper, we expand the study of \cite{reciprocal-curves} from curves to surfaces, i.e., we consider domains $D$ that are finitely generated $K$-algebras of dimension $2$. While we do not provide a complete characterization for the possible dimension of $\recip(D)$, we show several sufficient conditions for the dimension of $\recip(D)$ to be the maximal possible, i.e., $2$. Moreover, we do not require the base field $K$ to be algebraically closed, as was done in \cite{reciprocal-curves}; indeed, we shall see that the fact that $K$ is algebraically closed or not influences the dimension of closely-related domains, and sometimes the non-closure of $K$ helps in studying the dimension of $\recip(D)$ (see Example \ref{ex:circle} for the former phenomenon and Propositions \ref{prop:noKrational-deg23} and \ref{prop:noKrational-Xd} for the latter).

The structure of the paper is as follows. In Section \ref{sect:curves} we generalize the results of \cite{reciprocal-curves} by removing the condition that the base field $K$ is algebraically closed; the proof is mostly a repetition of the algebraically closed case, using a more refined look at the geometric results used. Section \ref{sect:general} contains some general results on the reciprocal complement of finitely generated $K$-algebras. Section \ref{sect:reducing} presents the main construction of the paper: we study localizations $\recip(D)[f]$ at a regular function $f$ by studying the zero set of $f$, allowing us to reduce the dimension of the varieties involved and (for surfaces) to apply the results proved for curves. This method requires passing from varieties defined over $K$ to varieties defined over the rational function field $K(t)$. In Section \ref{sect:applications} we show how this method can be applied to various classes of varieties; in particular, we completely solve the dimension problem for quadrics in $\aff^3$, i.e., find the dimension of $\recip(K[X,Y,Z]/(f))$ when $f$ is a polynomial of degree $2$ (Example \ref{ex:quadrics}). In the final Section \ref{sect:piano}, we study in greater detail the localizations of the reciprocal complement of the polynomial ring $K[X,Y]$, with particular emphasis on the integral closure of the localizations: when $K$ is an infinite field, we give a geometric proof of the fact that $\recip(K[X,Y])$ is not integrally closed (as opposed to the more computation-based proof given in \cite[Theorem 5.8]{EGL-RD-polinomyal}).

\section{Notation and preliminaries}
Let $K$ be a field with algebraic closure $\overline{K}$. By the \emph{affine space} $\aff^n_K$ of dimension $n$ over $K$ we mean the set of maximal ideals of the polynomial ring $K[X_1,\ldots,X_n]$. A point of $\aff^n_K$ can be represented by a $n$-tuple $(a_1,\ldots,a_n)\in\overline{K}^n$, and two $n$-tuples $(a_1,\ldots,a_n)$ and $(b_1,\ldots,b_n)$ represent the same point if and only if there is a $K$-automorphism $\sigma$ of $\overline{K}$ such that $b_i=\sigma(a_i)$ for every $i$.

Likewise, by the \emph{projective spaces} $\proj^n_K$ we mean the set of the homogeneous ideals of $K[X_0,X_1,\ldots,X_n]$ maximal among the ones properly contained in $(X_0,\ldots,X_n)$. Every point of $\proj^n_K$ can be represented by a $n$-tuple $[a_0:\cdots:a_n]$ of elements of $\overline{K}$, where at least one $a_i\neq 0$; two $n$-tuples $[a_0:\cdots:a_n]$ and $[b_0:\cdots:b_n]$ represent the same point if and only if there are $t\in\overline{K}\setminus\{0\}$ and a $K$-automorphism $\sigma$ of $\overline{K}$ such that $b_i=t\sigma(a_i)$ for every $i$. In particular, $[a_0:\cdots:a_n]=[ta_0:\cdots:ta_n]$ for every $t\in K$, $t\neq 0$.

In general, we identify points of $\aff^n_K$ and $\proj^n_K$ with their coordinate representation as an $n$-tuple. We also identify $\aff^n_K$ with the subset $\{[x_0:\cdots:x_n]\in\proj^n_K\mid x_0\neq 0\}\subset\proj^n_K$.

\medskip

We will use $\xx$ to denote a variety and $X_0,\ldots,X_n$ to denote indeterminates. When dealing with surfaces in $\aff^3$, we will use $X,Y,Z$ for the affine variables and $W$ for the homogenization variable.

If $g_1,\ldots,g_m\in K[X_1,\ldots,X_n]$, we denote by $\Vaff(g_1,\ldots,g_m)$ the variety defined by $g_1,\ldots,g_m$, i.e., $\Vaff(g_1,\ldots,g_m):=\{(a_1,\ldots,a_n)\in\aff^n_K\mid g(a_1,\ldots,a_m)=0\}$. If $\xx:=\Vaff(g_1,\ldots,g_m)\subset\aff^n_K$, the \emph{coordinate ring} of $\xx$ is
\begin{equation*}
K[\xx]:=\frac{K[X_1,\ldots,X_n]}{(g_1,\ldots,g_m)}.
\end{equation*}
Elements of $K[\xx]$ can be thought of as functions $f:\xx\longrightarrow\aff^1_K$; if $K$ is algebraically closed this just means that $f:\xx\longrightarrow K$. If $\xx=\Vaff(g_1,\ldots,g_m)$ is an affine variety and $f\in K[X_1,\ldots,X_n]$, we set
\begin{equation*}
\Vaff_\xx(f):=V(f)\cap\xx=\Vaff(f,g_1,\ldots,g_m).
\end{equation*}

If $g_1,\ldots,g_m\in K[X_0,\ldots,X_n]$ are homogeneous, then $g_1,\ldots,g_m$ also define a zero set in $\proj^n_K$, which we denote by $\Vproj(g_1,\ldots,g_m)$. 

If $\xx\subset\aff^n_K$ is an affine variety, we denote by $\overline{\xx}$ its \emph{projective closure}, i.e., the smallest closed set of $\proj^n_K$ (with respect to the Zariski topology) containing $\xx$. We set
\begin{equation*}
\xx_\infty:=\overline{\xx}\setminus\xx=\{[0:x_1:\cdots:x_n]\in\overline{\xx}\}.
\end{equation*}
We call the points $p\in\xx_\infty$ the \emph{points at infinity} for $\xx$. If $\xx$ is an affine curve, we say that $\xx$ is \emph{regular at infinity} if every $p\in\xx_\infty$ is a regular (i.e., nonsingular) point of $\xx$.

If $\xx$ is a variety and $\yy\subset\xx$ a subvariety, we denote by $\local_{\xx,\yy}$ the local ring of $\xx$ at $\yy$.

If $D$ is a finitely generated $K$-algebra, a \emph{realization} of $D$ is an affine variety $\xx\subseteq\aff^n_K$ such that $K[\xx]\simeq D$. 

\bigskip

All rings considered in this paper are integral domains, i.e., they do not contain zerodivisors.

If $D$ is an integral domain, we denote by $Q(D)$ the quotient field of $D$.

The \emph{reciprocal complement} of $D$ is the ring $\recip(D)$ generated by $1/x$, as $x$ ranges in $D\setminus\{0\}$. The reciprocal complement is again a domain with the same quotient field as $D$, and it is always local \cite[Theorem 2.6]{EGL-RD-polinomyal}. If $f\in D$, there is a prime ideal $\pp_f$ of $\recip(D)$ that is maximal among the ones not containing $1/f$ \cite[Proposition 2.7(1)]{guerrieri-recip}.

A \emph{valuation domain} is an integral domain $V$ such that, for every $x\in Q(V)$, either $x\in V$ or $x^{-1}\in V$. If $D$ is an integral domain and $K$ a field containing $D$, the \emph{Zariski space} $\Zar(K|D)$ is the set of all valuation rings containing $D$ and having quotient field $K$.

\section{Curves}\label{sect:curves}
In this section, we generalize the results of \cite{reciprocal-curves} from algebraically closed fields to arbitrary fields. 

The following result, which act as a bridge between the algebraic and the geometric world, is well-known (for the algebraically closed case, see e.g. \cite[Chapter 1, \textsection 6]{hartshorne}).
\begin{prop}
Let $K$ be a field. Let $\xx$ be a curve on $K$ and let $\overline{\xx}$ be its projective closure. Then, there is a bijection between the points of the normalization of $\overline{\xx}$ and the Zariski space $\Zar(Q(K[\xx])|K)$.
\end{prop}
\begin{proof}
There is an equivalence of categories between normal curves and finitely generated field extensions of transcendence degree $1$ over $K$ (see e.g. \cite[0BY1]{stacks-project}). Let $\yy$ be the curve corresponding to $Q(K[\xx])$; then, $\yy$ is the normalization of $\xx$. If $P$ is a point of $\yy$, then $\local_{\yy,P}$ is a DVR, and different points correspond to different valuation rings. If $V\in\Zar(Q(K[\xx])|K)$, then there are $a_1,\ldots,a_n\in V$ such that $Q(K[\xx])=Q(K[a_1,\ldots,a_n])$; then, the center of $V$ on the integral closure of $K[a_1,\ldots,a_n]$ corresponds to a point of $\yy$. The claim is proved.

\end{proof}

\begin{lemma}\label{lemma:ic-algebrefg}
Let $D$ be an integral domain, and $D\subseteq T\subseteq Q(D)$ be a ring such that $(D:T)\neq(0)$. Then, $\recip(D)=Q(D)$ if and only if $\recip(T)=Q(D)$. In particular, if $D$ is a finitely generated $K$-algebra for some field $K$ and $\overline{D}$ is its integral closure, then  $\recip(D)=Q(D)$ if and only if $\recip(\overline{D})=Q(D)$.
\end{lemma}
\begin{proof}
If $\recip(D)=Q(D)$, then $Q(D)=\recip(D)\subseteq\recip(T)\subseteq Q(D)$ and thus $\recip(T)=Q(D)$.

Suppose $\recip(T)=Q(D)$, and let $c\in(D:T)\setminus\{0\}$. For any $x\in Q(D)$, there are $y_1,\ldots,y_n\in T$ such that
\begin{equation*}
cx=\inv{y_1}+\cdots+\inv{y_n};
\end{equation*}
dividing the previous equality by $c$, we get
\begin{equation*}
x=\inv{cy_1}+\cdots+\inv{cy_n}\in\recip(D)
\end{equation*}
since $cy\in D$ for every $y\in T$. Hence $\recip(D)=Q(D)$, as claimed.

The ``in particular'' statement follows from the fact that, if $D$ is a finitely generated $K$-algebra, then $(D:\overline{D})\neq(0)$ \cite[Corollary 13.13]{eisenbud}.
\end{proof}

The following lemma is the geometric lynchpin of the proof of Proposition \ref{prop:zarminus1} and Theorem \ref{teor:ptoinf}; we isolate it since it has independent algebraic interest.
\begin{lemma}\label{lemma:zarminus1}
Let $K$ be a field, $F$ a finitely generated extension of $K$ of transcendence degree $1$. For every $V\in\Zar(F|K)$, $V\neq F$, the domain
\begin{equation*}
A:=\bigcap\{W\mid W\in\Zar(F|K), W\neq V\}
\end{equation*}
has quotient field $F$.
\end{lemma}
\begin{proof}
Each $W\in\Zar(F|K)$ is integrally closed and has quotient field $F$; thus, each of them contains the algebraic closure of $K$ in $F$. Hence, we can suppose without loss of generality that $K$ is algebraically closed in $F$.

Let $\xx$ be a normal projective curve over $K$ whose field of rational functions is $F$, and let $P$ be the point of $\xx$ corresponding to $V$. Take a divisor $tP$, with $t\inN^+$; for sufficiently large $t$, we have $\ell(W-tP)=0$, where $W$ is the canonical divisor of $\xx$ and $\ell(E)$ is the dimension of the $K$-vector space of the $f\in K[\xx]$ such that $\divv f+E\geq 0$. Let $g$ be the genus of $\xx$. By the Riemann-Roch theorem, for large $t$ we have
\begin{equation*}
\ell(tP)=\deg(tP)+1-q=t\deg P+1-g>1.
\end{equation*}
Therefore, there are linearly independent functions $f_1,f_2\in F$ such that $\divv f_i+kP\geq 0$ for $i=1,2$; in particular, at least one of them, say $f_1$, must be transcendental over $K$. Then, $f_1\in A$ since $f_1$ is regular on $\xx\setminus\{P\}$: therefore, the integral closure $B$ of $K[f_1]$ in $F$ is contained in $A$. However, $K(f_1)\subseteq F$ is an algebraic extension; therefore, $Q(B)=F$. Hence also $Q(A)=F$, as claimed.
\end{proof}

\begin{prop}\label{prop:zarminus1}
Let $K$ be a field and let $D$ be an integral finitely generated $K$-algebra of dimension $1$; let $F$ be the quotient field of $D$. Then, $\recip(D)=F$ if and only if $|\Zar(F|K)\setminus\Zar(F|D)|>1$.
\end{prop}
\begin{proof}%[Second proof of Proposition \ref{prop:zarminus1}]
Suppose first that $|\Zar(F|K)\setminus\Zar(F|D)|$ has a single element, say $V$. Then,
\begin{equation*}
K\subseteq D\cap V\subseteq\bigcap_{W\in\Zar(F|D)}W\cap V\subseteq\bigcap_{W\in\Zar(F|K)}W=K.
\end{equation*}
Therefore, if $f\in D\setminus K$ then $f\notin V$, and in particular $1/f\in V$. Thus $\recip(D)\subseteq V$, and in particular $\recip(D)\neq F$.

Conversely, suppose that $|\Zar(F|K)\setminus\Zar(F|D)|>1$; by Lemma \ref{lemma:ic-algebrefg}, we can suppose without loss of generality that $D$ is integrally closed. We show that no $Z\in\Zar(F|K)$ contains $\recip(D)$. If $Z\in\Zar(D)$ this is obvious since $1/f\notin Z$ whenever $0\neq f\in\mm_Z\cap D$ (the intersection is nonzero since $F=Q(D)$). Suppose $Z\notin\Zar(D)$, and let $A:=D\cap Z$: then,
\begin{equation}\label{eq:zarminus1-geo}
A=\bigcap_{W\in\Zar(F|D)\cup\{Z\}}W\supseteq\bigcap_{W\in\Zar(F|K)\setminus\{V\}}W,
\end{equation}
where $V\neq Z$ belongs to $\Zar(F|K)\setminus\Zar(F|D)$ (such an element exists by the hypothesis $|\Zar(F|K)\setminus\Zar(F|D)|>1$). The rightmost intersection of \eqref{eq:zarminus1-geo} has quotient field $F$ by Lemma \ref{lemma:zarminus1}; in particular, $\mm_Z\cap A\neq(0)$. If $f\in\mm_Z\cap A$, then $1/f\in\recip(A)\subseteq\recip(D)$ while $1/f\notin Z$. Hence $\recip(D)\nsubseteq Z$. It follows that $\recip(D)$ must be the whole $F$, as claimed.
\end{proof}

\begin{oss}
The second part of the previous proof can be interpreted geometrically in the following way: $\xx$ is a curve with (at least) two points at infinity; take one of them, say $P$ (corresponding to the valuation domain $Z$). There is a non-constant $f\in K[\xx]$ that is zero in $P$: then, $1/f\in\recip(K[\xx])$ but $1/f$ has a pole in $Q$, i.e., $1/f\notin Z$.
\end{oss}

We can express the previous proposition geometrically in the following way; note that this is just \cite[Theorem 2.1]{reciprocal-curves} without the hypothesis that the base field is algebraically closed.
\begin{teor}\label{teor:ptoinf}
Let $K$ be a field and $D$ be a finitely generated one-dimensional $K$-algebra. Then, the following are equivalent:
\begin{enumerate}[(i)]
\item $\recip(D)$ is a field;
\item if $\xx$ is a realization of $D$ that is regular at infinity, then $|\xx_\infty|=1$;
\item if $\xx$ is a realization of $D$ and $\nu:Y\longrightarrow\overline{\xx}$ is a normalization, then $|\nu^{-1}(\xx_\infty)|=1$.
\end{enumerate}
Moreover, if these conditions hold and $\xx_\infty=\{p\}$, then the integral closure of $\recip(K[\xx])$ is the local ring $\local_{\overline{\xx},p}$.
\end{teor}
\begin{proof}
The points of the normalization of $\overline{\xx}$ correspond to the points of the Zariski space $\Zar(Q(D)|K)$, excluding the quotient field $Q(D)$. Thus the condition $|\nu^{-1}(\xx_\infty)|=1$ (which reduces to $|\xx_\infty|=1$ if $\overline{\xx}$ is regular at infinity) is equivalent to $|\Zar(Q(D)|K)\setminus\Zar(Q(D)|D)|=1$, which by Proposition \ref{prop:zarminus1} is equivalent to $\recip(D)$ not being a field. Thus the three conditions are equivalent.

The last part of the statement follows as in the proof of \cite[Theorem 2.1]{reciprocal-curves}.
\end{proof}

We collect an interesting consequence of Theorem \ref{teor:ptoinf}.
\begin{prop}\label{prop:Dcaprecip}
Let $D$ be a one-dimensional finitely generated $K$-algebra. The following are equivalent:
\begin{enumerate}[(i)]
\item\label{prop:Dcaprecip:field} $\recip(D)$ is a field;
\item\label{prop:Dcaprecip:cap} $D\cap\recip(D)\neq K$;
\item\label{prop:Dcaprecip:subset} $D\subseteq\recip(D)$.
\end{enumerate}
\end{prop}
\begin{proof}
If $\recip(D)$ is a field, then it must be equal to the quotient field of $D$; hence, $D\subseteq\recip(D)$ and $D\cap\recip(D)=\recip(D)\neq K$. So \ref{prop:Dcaprecip:field} implies \ref{prop:Dcaprecip:cap} and \ref{prop:Dcaprecip:subset}. 

If $\recip(D)$ is not a field, then by Theorem \ref{teor:ptoinf} $\Zar(Q(D)|K)\setminus\Zar(Q(D)|D)=\{V\}$ for some valuation domain $V$ and $\recip(D)\subseteq V$. Therefore,
\begin{equation*}
D\cap\recip(D)\subseteq D\cap V\subseteq\bigcap_{\substack{W\in\Zar(Q(D)|K)\\ W\neq V}}W\cap V=\bigcap_{W\in\Zar(Q(D)|K)}W=K.
\end{equation*}
Thus \ref{prop:Dcaprecip:cap} $\Longrightarrow$ \ref{prop:Dcaprecip:field}.

Finally, \ref{prop:Dcaprecip:subset} $\Longrightarrow$ \ref{prop:Dcaprecip:cap} is obvious.
\end{proof}
\begin{comment}
\begin{prop}\label{prop:dimgeq-dim1}
Let $K\subseteq L$ be an algebraic field extension, and let $P\in\Spec(K[\XX])$ be a prime ideal. If $K[\XX]/P$ is one-dimensional and $PL[\XX]$ is a prime ideal, then
\begin{equation*}
\dim\recip\left(\frac{K[\XX]}{P}\right)\geq\dim\recip\left(\frac{L[\XX]}{PL[\XX]}\right)
\end{equation*}
\end{prop}
\begin{proof}
Let $D_K:=K[\XX]/P$ and $D_K:=L[\XX]/PL[\XX]$. If $\dim\recip(D_K)=1$ the claim follows since $\dim\recip(D_L)\leq\dim D_L=1$ by \marginparr{ref}. Suppose $\dim\recip(D_K)=0$. Since $PL[\XX]\cap K[\XX]=P$, we have an inclusion $D_K\subseteq D_L$; hence, $\recip(D_L)$ is a domain containing $Q(D_K)=\recip(D_K)$ and with quotient field $Q(D_L)$. Since the extension of quotient fields $Q(D_K)\subseteq Q(D_L)$ is algebraic, it follows that $\recip(D_L)=Q(D_L)$ has dimension $0$.
\end{proof}
\end{comment}

\begin{ex}\label{ex:circle}
Let $f(X,Y):=X^2+Y^2-1$, and consider the two domains
\begin{equation*}
D_\insR:=\frac{\insR[X,Y]}{(f)}\quad\text{and}\quad D_\insC:=\frac{\insC[X,Y]}{(f)},
\end{equation*}
and the corresponding curves $\xx_\insR$ and $\xx_\insC$. 

Then, $\xx_\insC$ has two points at infinity, namely $[0:1:i]$ and $[0:1:-i]$; therefore, $\recip(D_\insC)$ is a field.

On the other hand, $\xx_\insR$ has only one point at infinity, since the coordinates $[0:1:i]$ and $[0:1:-i]$ are conjugated over $\insR$. It follows that $\recip(D_\insR)$ has dimension $1$.

The difference between these two cases lies in the existence of the decomposition ($x,y$ are the images of $X,Y$)
\begin{equation*}
x=\frac{x}{x^2+y^2}=\inv{2(x+iy)}+\inv{2(x-iy)}
\end{equation*}
which expresses $x$ as a member of $\recip(D_\insC)$, but cannot be transferred as an equality in $\recip(D_\insR)$.
\end{ex}

\section{General results}\label{sect:general}
\begin{lemma}\label{lemma:recip-xi-1}
Let $K\subset F$ be fields and $x_1,\ldots,x_n\in F\setminus\{0\}$. Then,
\begin{equation*}
\recip(K[x_1,x_1^{-1},\ldots,x_n,x_n^{-1}])=K(x_1,\ldots,x_n).
\end{equation*}
\end{lemma}
\begin{proof}
We first note that $K(x_1,\ldots,x_n)$ is the quotient field of $D:=K[x_1,x_1^{-1},\ldots,x_n,x_n^{-1}]$, and thus $\recip(D)\subseteq K(x_1,\ldots,x_n)$. Take any $i=1,\ldots,n$. If $x_i$ is algebraic over $K$, then $K[x_i,x_i^{-1}]=K[x_i]=K(x_i)$. If $x_i$ is transcendental over $K$, then $K[x_i]$ is isomorphic to a polynomial ring over $K$, and $\recip(K[x_i,x_i^{-1}])=K(x_i)$ by \cite[Proposition 7]{GLO-egyptian}. Therefore, $\recip(D)$ contains $K(x_1),\ldots,K(x_n)$, and thus it is $K(x_1,\ldots,x_n)$.
\end{proof}

If $D$ is an integral domain, we denote by $\unit(D)$ the set of units of $D$.
\begin{prop}\label{prop:units}
Let $K$ be a field and $D$ be a finitely generated $K$-algebra. Then,
\begin{equation*}
\dim\recip(D)\leq\dim(D)-\trdeg(K(\unit(D))/K).
\end{equation*}
\end{prop}
\begin{proof}
Write $D:=K[a_1,\ldots,a_n]$, and let $u_1,\ldots,u_d\in\unit(D)$ be a transcendence basis of $K(\unit(D))$ over $D$. Then, $u_1^{-1},\ldots,u_d^{-1}\in D$ and thus 
\begin{equation*}
D=K[u_1,u_1^{-1},\ldots,u_d,u_d^{-1}][a_1,\ldots,a_n].
\end{equation*}
By Lemma \ref{lemma:recip-xi-1}, $\recip(K[u_1,u_1^{-1},\ldots,u_d,u_d^{-1}])=K(u_1,\ldots,u_d)$; hence, 
\begin{equation*}
\recip(D)=\recip(K(u_1,\ldots,u_d)[a_1,\ldots,a_n]).
\end{equation*}
The ring $A:=K(u_1,\ldots,u_d)[a_1,\ldots,a_n]$ is a localization of $D$; therefore, for every maximal ideal $M$ of $A$ there is a prime ideal $P$ of $D$ such that $M=PA$, and $A/M$ is isomorphic to $D/P$. The residue field $A/M$ is algebraic over $K(u_1,\ldots,u_d)$, and thus the transcendence degree of $D/P$ over $K$ is $d=\trdeg(K(\unit(D))/K)$. It follows that $P$ has height $\dim(D)-d$, and so $\dim(A)=\dim(D)-d$.

By \cite[Theorem 3.2]{guerrieri-recip}, since $A$ is a finitely generated $K(u_1,\ldots,u_d)$-algebra we have
\begin{equation*}
\dim\recip(D)=\dim\recip(A)\leq\dim A=\dim(D)-\trdeg(K(\unit(D))/K),
\end{equation*}
as claimed.
\end{proof}

\begin{cor}\label{cor:units-ac}
Let $K$ be an algebraically closed field and $D$ be a finitely generated $K$-algebra. If $\unit(D)\neq K\setminus\{0\}$, then
\begin{equation*}
\dim\recip(D)<\dim(D).
\end{equation*}
\end{cor}
\begin{proof}
Let $u\in\unit(D)\setminus K$. Since $K$ is algebraically closed, $u$ is transcendental over $K$, and thus $\trdeg(K(\unit(D))/K)\geq 1$. The claim follows from Proposition \ref{prop:units}.
\end{proof}

\begin{cor}
Let $K$ be an algebraically closed field, $X_1,\ldots,X_n$ a set of independent indeterminates, $g_1,g_2\in K[X_1,\ldots,X_n]\setminus K$ and $d\in K$. Let
\begin{equation*}
f(X_1,\ldots,X_n):=g_1(X_1,\ldots,X_n)g_2(X_1,\ldots,X_n)-d
\end{equation*}
and $D:=K[X_1,\ldots,X_n]/(f)$. If $f$ is irreducible, then
\begin{equation*}
\dim\recip(D)<\dim D.
\end{equation*}
\end{cor}
\begin{proof}
Since $f$ is irreducible, $D$ is an integral domain and $d\neq 0$; moreover, the images of $g_1$ and $g_2$ are units of $D$, since their product is $d\in K$. The claim follows from Corollary \ref{cor:units-ac}.
\end{proof}

In general, we can have $\dim\recip(D)<\dim D$ even if all the units of $D$ are constant, as the following example shows.

\begin{ex}
Let $\xx$ be a nonsingular affine curve. Suppose that $|\xx_\infty|=1$ and that the point at infinity is regular for $\overline{\xx}$. Suppose that there is a maximal ideal $\mm$ of $K[\xx]$ such that $\mm^k$ is not principal for every $k$, i.e., such that $[\mm]$ is not a torsion element of the Picard group of $K[\xx]$, and let $P$ be the point of $\xx$ corresponding to $\mm$. Then, $\xx':=\xx\setminus\{P\}$ is an affine curve with two points at infinity, and thus by Theorem \ref{teor:ptoinf} $\recip(K[\xx'])$ is a field.

Suppose $f\in K[\xx']$ is a unit. Since $P$ is a regular point of $\xx$, the local ring $\local_{\xx,P}$ is a DVR, and thus one of $f$ and $f^{-1}$ belongs to $\local_{\xx,P}$; suppose without loss of generality that $f\in\local_{\xx,P}$, and thus $f\in K[\xx]=K[\xx']\cap\local_{\xx,P}$. By construction, $f$ is not contained in any maximal ideal $\mathfrak{n}\neq\mm$ of $K[\xx]$. If $f\in\mm$, then there is another maximal ideal $\mm'$ of $K[\xx]$ such that $f\in\mm'$ (otherwise $\mm^k=(f)$ for some $k$, a contradiction). This implies that $\mm'K[\xx']\neq K[\xx']$ and thus $f$ is not a unit of $K[\xx']$, a contradiction. Hence $f\notin\mm$ and thus $f$ is a unit of $K[\xx]$. Therefore, $f$ is a rational function on $\overline{\xx}$ without zeros nor poles on $\xx$; since $\xx_\infty$ is a single point, $f$ must be constant, i.e., $f\in K$.

Therefore, $K[\xx']$ has no units outside $K$, but $0=\dim\recip(K[\xx'])<\dim K[\xx']=1$.
\end{ex}

Let $\xx$ be an affine curve with projective closure $\overline{\xx}$ that is regular at infinity. Theorem \ref{teor:ptoinf} can be rephrased by saying that $\dim\recip(K[\xx])=1$ if and only if $\xx_\infty$ is irreducible.

This result cannot be generalized in this way to surfaces, because when the dimension is greater than one two realizations $\xx$ and $\xx'$ of the same domain $D$ may be such that $\xx_\infty$ is irreducible but $\xx'_\infty$ is not, even if the projective closures $\overline{\xx}$ and $\overline{\xx'}$ are nonsingular. For example, if $D=K[X,Y]$ is the polynomial ring, then $\xx:=\aff^2$ is a realization and $\xx_\infty$ is irreducible (since it is the line at infinity); on the other hand, $\xx':=\Vaff(XY-Z)\subset\aff^3$ is another realization of $D$, but $\xx_\infty$ is not reducible as it splits into the two lines $\Vproj(X,W)$ and $\Vproj(Y,W)$.% Yet, obviously, $K[\xx]\simeq K[\xx']$, so the reciprocal complements are isomorphic and thus have the same dimension.

However, we have the following weaker result.

\begin{prop}\label{prop:dim2-infty-irred}
Let $\xx$ be an affine surface such that $\xx_\infty$ is irreducibile and contains nonsingular points of $\xx$. Then:
\begin{enumerate}[(a)]
\item\label{prop:dim2-infty-irred:intersec} $K[\xx]\cap\recip(K[\xx])=K$;
\item\label{prop:dim2-infty-irred:dim} $\dim\recip(K[\xx])\geq 1$.
\end{enumerate}
\end{prop}
\begin{proof}
Let $\Omega$ be the set of nonsingular points of $\overline{\xx}$. Since $\Omega\cap \xx_\infty\neq\emptyset$, it must be an open set of $\xx_\infty$, and thus only finitely many points of $\xx_\infty$ are singular in $\overline{\xx}$. Moreover, there is an affine open subset $\Omega'$ of $\Omega$ meeting $\xx_\infty$; then, $K[\Omega']$ is integrally closed and $\local_{\Omega',Y}=\local_{\overline{\xx},\xx_\infty}$ is local, Noetherian, integrally closed and one-dimensional, hence a DVR \cite[Proposition 9.2]{atiyah}.

Suppose $f\in K[\xx]$: we claim that $1/f\in\local_{\overline{\xx},\xx_\infty}$. If $f\in K$ this is obvious; if $f\notin K$, then $f$ cannot be regular on $\xx_\infty$ (otherwise it would be regular on the whole $\xx\cup\xx_\infty=\overline{\xx}$, against the fact that $f$ is not constant) and so $f\notin\local_{\overline{\xx},\xx_\infty}$. As $\local_{\overline{\xx},\xx_\infty}$ is a DVR, $1/f\in\local_{\overline{\xx},\xx_\infty}$.

Therefore, $\recip(K[\xx])\subseteq\local_{\overline{\xx},\xx_\infty}$, and so $\recip(K[\xx])$ is not a field; moreover,
\begin{equation*}
K[\xx]\cap\recip(K[\xx])\subseteq K[\xx]\cap\local_{\overline{\xx},\xx_\infty}=K.
\end{equation*}
The claim is proved.
\end{proof}

\section{Reducing the dimension}\label{sect:reducing}
Let $\xx$ be a surface and $f\in K[\xx]$. We can associate to $f$ two subsets of $\xx$: the closed set $\Vaff_\xx(f):=\{x\in \xx\mid f(x)=0\}$ and its complement $D_\xx(f):=\xx\setminus \Vaff_\xx(f)$, which is an open set. We always have
\begin{equation*}
K[D_\xx(f)]=K[\xx][1/f].
\end{equation*}
By \cite[Corollary 2.6]{guerrieri-recip}, the previous equality can be transported to reciprocal complements as
\begin{equation*}
\recip(K[D_\xx(f)])=\recip(K[\xx][1/f])=\recip(K[\xx])[f].
\end{equation*}
Therefore, the study of localizations of $\recip(K[\xx])$ is equivalent to studying the reciprocal complements of the (principal) open sets of $\xx$. However, this method still forces us to study varieties of dimension $2$.

To pass from dimension $2$ to dimension $1$ we would like to consider the variety $\Vaff_\xx(f)$ and the reciprocal complement $\recip(K[\Vaff_\xx(f)])$. While we can't do this passage, we can introduce a new variety that is very close to $\Vaff_\xx(f)$ and whose reciprocal complement is a localization to $\recip(K[\xx])$. The drawback is that we need to change the base field.
\begin{defin}
Let $K$ be a field, and let $\xx:=\Vaff(g_1,\ldots,g_m)\subset\aff^n_K$ be a variety over $K$. Let $t$ be an indeterminate, and $f\in K[X_1,\ldots,X_n]$. The \emph{$f$-transform} of $\xx$ is the variety
\begin{equation*}
\Vaff(g_1,\ldots,g_m,f-t)\subset\aff^n_{K(t)}.
\end{equation*}
\end{defin}

We prove the properties of the $f$-transform that we will be using; we premit a lemma.
\begin{lemma}\label{lemma:alg-const}
Let $K$ be a field, $g_1,\ldots,g_m\in K[X_1,\ldots,X_n]$, and let $\xx:=\Vaff(g_1,\ldots,g_m)\subset\aff^n_K$. Let $f\in K[X_1,\ldots,X_n]$. Then, $f+P\in K[\xx]$ is constant over $\xx$ if and only if $f+P$ is algebraic over $K$.
\end{lemma}
\begin{proof}
Suppose that $\overline{f}:=f+P$ is algebraic over $K$: then, $\lambda(f)=0$ for some nonzero irreducible polynomial $\lambda\in K[t]$. Hence, $\lambda(\overline{f}(p))=0$ for every $p\in\xx$, i.e., each $\overline{f}(p)$ is a zero of $\lambda$. Since $\lambda$ is irreducible, its zeros are conjugated over $K$; hence all these $\overline{f}(p)$ represent the same point in $\aff^1_K$. Thus $\overline{f}$ is constant.

Conversely, if $\overline{f}$ is constant, then $f(p)$ can be represented by some $t\in\overline{K}$; if $\lambda$ is the minimal polynomial of $t$ over $K$ then $\lambda(\overline{f})=0$ and $\overline{f}$ is algebraic over $K$.
\end{proof}

\begin{prop}\label{prop:transf}
Let $K$ be a field, $X_1,\ldots,X_n,t$ be independent indeterminates over $K$, $P=(g_1,\ldots,g_m)$ a prime ideal of $K[X_1,\ldots,X_n]$; let $\xx:=\Vaff(g_1,\ldots,g_m)\subset\aff^n$. Let $f\in K[X_1,\ldots,X_n]$ be such that $f+P\in K[\xx]$ is not constant, and let $\pp_f$ be the largest prime ideal of $\recip(K[\xx])$ not containing $(f+P)^{-1}$. Then, the following hold.
\begin{enumerate}[(a)]
\item\label{prop:transf:coord} $\displaystyle{K(t)[\transf_\xx(f)]\simeq\frac{K(t)[X_1,\ldots,X_n]}{(g_1,\ldots,g_m,f-t)}}$.
\item\label{prop:transf:recip} $\recip(K[\xx])_{\pp_f}\simeq\recip(K(t)[\transf_\xx(f)])$.
\item\label{prop:transf:irrid} $\transf_\xx(f)$ is irreducible.
\item\label{prop:transf:dim} If $f$ is not a unit of $K[\xx]$, then $\dim\transf_\xx(f)=\dim \xx-1$.
\end{enumerate}
\end{prop}
\begin{proof}
\ref{prop:transf:coord} follows immediately from the definition of $\transf_\xx(f)$. 

\ref{prop:transf:recip} Let $\overline{f}:=f+P$ and let $x_i$ be the image of $X_i$ in $A:=K[\xx]$ for $i=1,\ldots,n$. By \cite[Corollary 2.6 and Proposition 2.7]{guerrieri-recip}, $\recip(D)_{\pp_f}=\recip(S^{-1}K[X,Y])$, where $S=\{\overline{f}^k\mid k\inN\}$. We have
\begin{equation*}
S^{-1}A=K[x_1,\ldots,x_n][1/\overline{f}]=K[\overline{f},1/\overline{f},x_1,\ldots,x_n]=K[\overline{f},1/\overline{f}][x_1,\ldots,x_n].
\end{equation*}
By Lemma \ref{lemma:alg-const}, $\overline{f}$ is transcendental over $K$, and thus $\recip(K[\overline{f},1/\overline{f}])=K(\overline{f})\simeq K(t)$, where $t$ is an indeterminate over $K$ (Lemma \ref{lemma:recip-xi-1}); hence,
\begin{equation*}
\recip(K[\overline{f},1/\overline{f}][x_1,\ldots,x_n])=\recip(K(f)[x_1,\ldots,x_n]).
\end{equation*}
However, setting $L=K(f)$ and taking an indeterminate $t$, we have
\begin{equation*}
K(f)[x_1,\ldots,x_n]\simeq\frac{L[X_1,\ldots,X_n]}{(g_1,\ldots,g_m,f-t)}.
\end{equation*}
The claim is proved. Moreover, the calculation shows that $K(t)[\transf_\xx(f)]$ is isomorphic to a localization of $K[\xx]$, and thus is a domain; in particular, $\transf_\xx(f)$ is irreducible and \ref{prop:transf:irrid} holds.

\ref{prop:transf:dim} The height of the ideal $P=(g_1,\ldots,g_m)$ does not change between $K[X_1,\ldots,X_n]$ and $K(t)[X_1,\ldots,X_n]$, and thus the dimension of $K[\xx]$ and $B:=K(t)[X_1,\ldots,X_n]/(g_1,\ldots,g_m)$ is the same. Moreover, since the image of $f$ is not a unit in $K[\xx]$, the image $f_1$ of $f-t$ is not a unit of $B$.

The ring $B$ is Noetherian and catenarian; therefore, $\dim\transf_\xx(f)=\dim (B/(f_1))=\dim B-1=\dim\xx-1$. 
\end{proof}

In particular, if $\xx$ is a surface we have successfully reduced the study of the localization $\recip(K[\xx])[f]$ to the study of the reciprocal complement of a one-dimensional variety -- namely, $\transf_\xx(f)$. Since $\transf_\xx(f)$ is no more a variety of $K$ but over $K(t)$, which is not algebraically closed, the results of Section \ref{sect:curves} are needed even when the starting field $K$ is algebraically closed.

In this article, we are mainly interested in studying conditions under which $\dim\recip(K[\xx])=2$. Since $\recip(K[\xx])$ is local and the dimension is at most $2$, to prove that $\dim\recip(K[\xx])=2$ it is enough to find \emph{one} $f$ such that $\dim\recip(K(t)[\transf_\xx(f)])=1$, and thus we need to find a function $f$ such that $\transf_\xx(f)$ has a unique point at infinity.

We first note that, if $p=[0:x_1:\cdots:x_n]\in\proj^n_K$ is a point at infinity of $\Vaff_\xx(f)$, then $[0:x_1:\cdots:x_n]\in\proj^n_{K(t)}$ is a point at infinity of $\transf_\xx(f)$, since the homogeneous components of maximal degree of $f$ and $f-t$ are the same. With a slight abouse of notation, we consider these two points with the same coordinates in $\proj^n_K$ and $\proj^n_{K(t)}$ as the same point.

We prove a more general criterion about solutions of system of equation in $K$ and $K(t)$. We use the following notation: if $f_1,\ldots,f_m\in K[X_1,\ldots,X_m]$ and $K\subseteq L$, $\Vaff^L(f_1,\ldots,f_m)$ is the closed set $\aff^n_L$ associated to $f_1,\ldots,f_m$ (and similarly for the projective space).%If $f_1,\ldots,f_n\in K[X_1,\ldots,X_m]$, a solution of $f_1=\cdots=f_n=0$ in $\aff^n_K$ is just a point $p\in V(f_1,\ldots,f_n)$, i.e., a point $p\in\aff^m_K$ such that $f_1(p)=\cdots=f_n(p)=0$. Similarly, if $F_1,\ldots,F_n\in K[X_1,\ldots,X_m]$ are homogeneous then a solution of $F_1=\cdots=F_n=0$ in $\proj^{m-1}_K$ is just a point $p\in V(F_1,\ldots,F_n)\subset\proj^{m-1}_K$.
\begin{lemma}\label{lemma:numsol-Kt}
Let $K$ be a field, $t,X_1,\ldots,X_n$ indeterminates.
\begin{enumerate}[(a)]
\item If $f_1,\ldots,f_m\in K[X_1,\ldots,X_n]$ and $\Vaff^K(f_1,\ldots,f_m)$ is finite, then $|\Vaff^K(f_1,\ldots,f_m)|=|\Vaff^{K(t)}(f_1,\ldots,f_m)|$.
\item If $F_1,\ldots,F_m\in K[X_0,\ldots,X_n]$ are homogeneous and $\Vproj^K(F_1,\ldots,F_m)$ is finite, then $|\Vproj^K(F_1,\ldots,F_m)|=|\Vproj^{K(t)}(F_1,\ldots,F_m)|$.
\end{enumerate}
\end{lemma}
\begin{proof}
Let $\xx:=\Vaff(f_1,\ldots,f_m)\subset\aff_K^n$ and $A:=K[X_1,\ldots,X_n]$, $B:=K(t)[X_1,\ldots,X_n]$. The extension $A\subset B$ is flat, since $A[t]=K[t,X_1,\ldots,X_n]$ is free over $A$ and $B$ is a localization of $A[t]$ (hence flat). Moreover, if $P$ is a prime ideal of $A$, then $PA[t]$ is prime in $A[t]$ and thus $PB$ is prime in $B$.

Let $I:=\rad(f_1,\ldots,f_m)$. Then, $I$ is a radical ideal of height $n$ and the maximal ideals $M_1,\ldots,M_s$ containing $I$ are exactly the points of $\xx$: thus, $I=M_1\cap\cdots\cap M_s$. Since $A\subset B$ is flat, by \cite[Theorem 7.4]{matsumura} $IB=(M_1\cap\cdots\cap M_s)B=M_1B\cap\cdots\cap M_sB$. Each $M_iB$ is a maximal ideal of $B$, and thus the points of $\Vaff(f_1,\ldots,f_m)\subset\aff_{K(t)}^n$ are $M_1B,\ldots,M_sB$. In particular, the number of points is the same.

For the homogeneous case, we consider $I:=\rad(F_1,\ldots,F_n)$: then, $I$ is an homogeneous radical ideal, and thus its minimal primes are homogeneous \cite[Theorem 13.7(i)]{matsumura}; therefore, $I=P_1\cap\cdots\cap P_s$, where each $P_i$ is maximal among the homogeneous ideals of $A$ properly contained in $(X_1,\ldots,X_m)$. As above, $IB=P_1B\cap\cdots\cap P_sB$ and the number of points is the same.
\end{proof}

\begin{prop}\label{prop:transf-infty}
Let $f,g\in K[X,Y,Z]$ be non-constant polynomials with leading homogeneous components $f_1,g_1$, and let $\xx:=\Vaff(g)\subset\aff^3$. Suppose that $\Vproj(f_1,g_1)\subset\proj^2_K$ is a singleton. Then, $|\Vaff_\xx(f)_\infty|=|\transf_\xx(f)_\infty|=1$, and the point at infinity is the same.
\end{prop}
\begin{proof}
Let $F,G$ be the homogenizations of $f$ and $g$, respectively. Since $\Vaff_\xx(f)=\Vaff(f,g)$, the closure $\overline{\Vaff_\xx(f)}$ is contained into $\Vproj(F,G)$. Let $W$ be the homogenization variable: when $W=0$, that is, if we consider $\Vaff_\xx(f)_\infty$, we have $F(x,y,z)=0$ if and only if $f_1(x,y,z)=0$, and similarly for $G$ and $g_1$. Hence, $\overline{\Vaff_\xx(f)}\subseteq\Vaff_\xx(f)\cup\Vproj(f_1,g_1,W)$.

If $p:=[x:y:z]$ is the unique point of $\Vproj(f_1,g_1)\subset\proj^2_K$, then $\overline{\Vaff_\xx(f)}\subseteq \Vaff_\xx(f)\cup\{[0:x:y:z]\}$. Since $\Vaff_\xx(f)$ is not closed in $\proj^2$, we must have equality, and thus $\Vaff_\xx(f)_\infty=\{[0:x:y:z]\}$ is a single point.

By Lemma \ref{lemma:numsol-Kt}, $\Vproj^{K(t)}(f_1,g_1)\subset\proj^2_{K(t)}$ is a singleton too. Since the leading homogeneous component of $f-t$ is still $f_1$, the reasoning above also holds for $\transf_\xx(f)$, and $\transf_\xx(f)_\infty=\{[0:x:y:z]\}$ is a singleton.
\end{proof}

The previous lemma allows us to concretely analyze some localizations of $\recip(K[X])$; the following example will be generalized in different ways in Examples \ref{ex:cubics} and \ref{ex:fermat}.
\begin{ex}
Let $K$ be a field of characteristic $\neq 3$, and let $\xx\subseteq\aff^3_K$ be the cubic hypersurface defined by $X^3+Y^3+Z^3+1=0$, and let $f(X,Y,Z):=X+Y$. We note that, if $K$ is algebraically closed, $\Vaff_\xx(f)$ is the union of three lines (namely, $Z=1$, $Z=\omega$ and $Z=\omega^2$, where $\omega$ is a cubic root of unity). The system
\begin{equation}\label{eq:esempio-numsol}
\left\{\begin{aligned}
& X^3+Y^3+Z^3=0\\
& X+Y=0
\end{aligned}\right.
\end{equation}
has a single solution in $\proj^2_K$, namely $[1:-1:0]$. By Proposition \ref{prop:transf-infty}, $p:=[0:1:-1:0]$ is the unique point at infinity of $\Vaff_\xx(f)$ and $\transf_\xx(f)$.

We claim that $p$ is regular in $\overline{\transf_\xx(f)}$. Indeed, taking homogeneous coordinates, the Jacobian matrix of the curve is
\begin{equation*}
\begin{pmatrix}
3W^2 & 3X^3 & 3Y^2 & 3Z^2\\
-t & 1 & 1 & 0
\end{pmatrix}
\end{equation*}
which calculated in $p=[0:1:-1:0]$ gives
\begin{equation*}
\begin{pmatrix}
0 & 3 & 3 & 0\\
-t & 1 & 1 & 0
\end{pmatrix}.
\end{equation*}
As this matrix has rank $2$, $p$ is regular. By Theorem \ref{teor:ptoinf} $\recip(K(t)[\transf_X(f)])$ is nontrivial. By Proposition \ref{prop:transf}\ref{prop:transf:recip}, $\recip(K(t)[\transf_\xx(f)])\simeq\recip(K[\xx])_{\pp_f}$, and thus $\dim\recip(K[\xx])_{\pp_f}=1$. By Proposition \ref{prop:dim2-infty-irred}, $f\notin\recip(K[\xx])$, and thus $\recip(K[\xx])_{\pp_f}$ is a proper localization of $\recip(K[\xx])$. Since reciprocal complements are local, it follows that $\dim\recip(K[\xx])=2$.
\end{ex}

In the previous example, we showed explicitly that the point at infinity of $\transf_\xx(f)$ was nonsingular. The following proposition gives a general criterion.
\begin{prop}\label{prop:transf-nonsing}
Let $\xx\subset\aff^n_K$ be a variety, and let $f\in K[X_1,\ldots,X_n]$ be such that the image of $f$ in $K[\xx]$ is not constant. Let $p\in \Vaff_\xx(f)_\infty$, and suppose that $p$ is nonsingular in $\overline{\xx}$.
\begin{enumerate}[(a)]
\item\label{prop:transf-nonsing:VXF->transf} If $p$ is nonsingular in $\overline{\Vaff_\xx(f)}$, then $p$ is nonsingular in $\overline{\transf_\xx(f)}$.
\item\label{prop:transf-nonsing:notlinear} If $f$ is not linear, then $p$ is nonsingular in $\overline{\transf_\xx(f)}$ if and only if it is nonsingular in $\overline{\Vaff_\xx(f)}$.
\item\label{prop:transf-nonsing:linear} If $f$ is linear and $\xx$ is an hypersurface, then $p$ is nonsingular in $\overline{\transf_\xx(f)}$.
\end{enumerate}
\end{prop}
\begin{proof}
Let $\xx=\Vaff(g_1,\ldots,g_m)$, and let $J_1$ be the Jacobian matrix of $G_1,\ldots,G_m$, the homogenizations of $g_1,\ldots,g_m$ with respect to $X_0$. Let $F$ be the homogenization of $f$; then, the homogenization of $f-t$ is $F-tX_0^d$, where $d$ is the degree of $F$. Hence, the Jacobian matrices of $G_1,\ldots,G_m$ and $G_1,\ldots,G_m,F-t$ are, respectively,
\begin{equation*}
J:=\begin{pmatrix}
J_1\\
\nabla F
\end{pmatrix}\quad\text{and}\quad J':=\begin{pmatrix}
J_1\\
\nabla(F-tX_0^d)
\end{pmatrix}.
\end{equation*}
If $d>1$, then 
\begin{equation*}
\frac{\partial(F-tX_0^d)}{\partial {X_0}}(p)=\frac{\partial F}{\partial X_0}(p)-dX_0^{d-1}(p)=\frac{\partial F}{\partial X_0}(p)
\end{equation*}
since $p$ is at infinity and thus is $X_0$-component is $0$. Hence, $J(p)=J'(p)$ and in particular they have the same rank. Since $\dim\transf_\xx(f)=\dim \xx-1=\dim \Vaff_\xx(f)$ (Proposition \ref{prop:transf}\ref{prop:transf:dim}) and the nonsingularity of $p$ only depends on the rank of the Jacobian matrix, it follows that $p$ is nonsingular in $\overline{\transf_\xx(f)}$ if and only if it is nonsingular in $\overline{\Vaff_\xx(f)}$. Hence \ref{prop:transf-nonsing:notlinear} holds.

Suppose $d=1$. Then, the Jacobian matrices of $G_1,\ldots,G_m$ and $G_1,\ldots,G_m,F-t$ are, respectively,
\begin{equation*}
J:=\begin{pmatrix}
J_1\\
\nabla F
\end{pmatrix}\quad\text{and}\quad J':=\begin{pmatrix}
J_1\\
\nabla(F-tX_0)
\end{pmatrix}=:=\begin{pmatrix}
J_1\\
\nabla(F)-te_{X_0}
\end{pmatrix}
\end{equation*}
where $e_{X_0}$ is the basic component relative to $X_0$. If $p$ is nonsingular in $\overline{\Vaff_\xx(f)}$, then $\nabla F(p)$ is not a linear combination of the rows $\nabla G_1(p),\ldots\nabla G_m(p)$ of $J_1(p)$. If $p$ were to be singular in $\transf_\xx(f)$, on the other hand, then $(\nabla F-te_{X_0})(p)$ would be a linear combination of these rows, i.e., there would be $\lambda_1,\ldots,\lambda_m\in K(t)$, not all zeros, such that
\begin{equation}\label{eq:prop:transf}
\nabla(F-te_{X_0})(p)=\left(\lambda_1\nabla G_1+\cdots+\lambda_m\nabla G_m\right)(p)=0.
\end{equation}
Reducing denominators, we can suppose without loss of generality that $\lambda_1,\ldots,\lambda_m\in K[t]$ and not all of them are divisible by $t$. In particular, \eqref{eq:prop:transf} would hold also calculated in $t=0$; however, in this case we would have an equation of linear dependence for $\nabla F$, against the nonsingularity of $p$ in $\overline{\Vaff_\xx(f)}$. Hence $p$ is also nonsingular in $\overline{\transf_\xx(f)}$. Putting together this and the previous case, \ref{prop:transf-nonsing:VXF->transf} holds.

\ref{prop:transf-nonsing:linear} Suppose that $d=1$ and that $\xx$ is an hypersurface: then, $n=1$. Write $f=a_1X_1+\cdots+a_nX_n+b$: then, the Jacobian matrix of $G_1,F-tX_0$ at $p$ is
\begin{equation*}
J(p):=\begin{pmatrix}
\partial_{X_0}G(p) & \partial_{X_1}G(p) & \cdots & \partial_{X_n}G(p)\\
 b-t & a_1 & \cdots & a_m
\end{pmatrix}.
\end{equation*}
The first row of $J(p)$ is nonzero since $p$ is regular in $\xx$; let $i$ be such that $a_i\neq 0$. If $p$ is not regular in $\overline{\transf_\xx(f)}$, the two rows of $J(p)$ must be proportional, and the ratio between the two rows must be $\partial_{X_i}G(p)/a_i\in K$. However, $\partial_{X_0}G(p)/(b-t)\notin K$ as $\partial_{X_0}G(p),b\in K$ while $t\notin K$, a contradiction. Hence $J(p)$ has rank $2$ and $p$ is nonsingular in $\overline{\transf_\xx(f)}$.
\end{proof}

The previous results can be summarized in the following criterion.
\begin{prop}\label{prop:affine-irred-VXFreg}
Let $\xx:=\Vaff(f)\subset\aff^3_K$ be an affine surface such that $\xx_\infty$ is irreducible. Suppose that there is a point $p\in\xx_\infty$ and a function $g\in K[X,Y,Z]$ such that:
\begin{itemize}
\item $p$ is a nonsingular point of $\overline{\xx}$;
\item $\overline{\Vaff_\xx(g)}\cap\xx_\infty=\{p\}$;
\item $g$ is linear; \emph{or} $p$ is a regular point of $\overline{\Vaff_\xx(g)}$.
\end{itemize}
Then, $\dim\recip(K[\xx])=2$.
\end{prop}
\begin{proof}
By construction, $g$ is not a constant function, and thus by Proposition \ref{prop:dim2-infty-irred}\ref{prop:dim2-infty-irred:intersec} $g\notin\recip(K[\xx])$. Therefore, $\recip(K[\xx])[g]$ is a proper localization of $\recip(K[\xx])$. By Proposition \ref{prop:transf}\ref{prop:transf:recip}, $\recip(K[\xx])[g]$ is isomorphic to the reciprocal complement of $K(t)[\transf_\xx(g)]$. 

Moreover, by hypothesis $g$ is not constant over $\xx_\infty$, and thus the leading homogeneous components $f_1,g_1$ of $f$ and $g$ are coprime, hence $\Vproj(f_1,g_1)$ is finite. By hypothesis, $\overline{\Vaff_\xx(g)}$ has a unique point at infinity, namely $p$; by Proposition \ref{prop:transf-infty}, $p$ is the unique point at infinity of $\transf_\xx(g)$, and $p$  is nonsingular by Proposition \ref{prop:transf-nonsing} (by \ref{prop:transf-nonsing:linear} if $f$ is linear, by \ref{prop:transf-nonsing:VXF->transf} if $p$ is regular in $\overline{\Vaff_\xx(f)}$). By Theorem \ref{teor:ptoinf}, $\dim\recip(K(t)[\transf_\xx(f)])=1$. Since $\recip(K[\xx])$ is local, its dimension must be strictly greater than $1$, and thus is equal to $2$.
\end{proof}

\begin{oss}
An application of Riemann-Roch's theorem guarantees that there is always a function $\phi\in K(\xx)$ such that $p$ is the only zero of $\phi$ on $\xx_\infty$. However, it is not clear if we can also chose $\phi$ for $p$ to be regular in the zero locus of $\phi$.
\end{oss}

\section{Applications and special cases}\label{sect:applications}
In this section, we apply the results above to specific classes of surfaces, defined by specific classes of polynomials. The first result is relevant when $K$ is not algebraically closed.
\begin{prop}\label{prop:noKrational-deg23}
Let $K$ be an infinite field, and let $f\in K[X,Y,Z]$ be a polynomial of degree $2$ or $3$; let $\xx:=\Vaff(f)\subset\aff^3_K$. Suppose that:
\begin{itemize}
\item $\xx_\infty$ is irreducible;
\item $\xx_\infty$ contains non-singular points of $\xx$;
\item $\xx_\infty$ does not contain any $K$-rational point.
\end{itemize}
Then, $\dim\recip(K[\xx])=2$.
\end{prop}
\begin{proof}
Let $\xx_{ns}$ be the set of nonsingular points of $\xx$: then, $\xx_{ns}$ is open, and thus $\xx_{ns}\cap\xx_\infty$ is an open set of $\xx_\infty$. The intersection is nonempty by hypothesis; since $K$ is infinite, it follows that all but finitely many points of $\xx_\infty$ are nonsingular. In particular, there are $a,b,c\in K$, not all zero, such that the line defined by $g(X,Y,Z):=aX+bY+cZ$ does not meet any point of $\xx_\infty$ that is singular for $\xx$.

Let $f_1$ be the leading homogeneous component of $f$, and consider the system $f_1=g=0$ in $\proj^2_K$. We can suppose without loss of generality that $c=1$: then, these points are the solutions of $h(X,Y):=f_1(X,Y,-aX-bY)$, and $h$ is a homogeneous polynomial of the same degree as $f$. Since $h$ has no $K$-rational solutions, and since its degree is at most $3$, it must be irreducible; hence its factors in the algebraic closure $\overline{K}$ must be conjugated over $K$. Therefore, the coordinates of the points in $\proj^2_{\overline{K}}$ that solve $f_1=g=0$ are conjugated over $K$, and thus they represent the same point in $\proj^2_K$.

By Proposition \ref{prop:transf-infty}, the $g$-transform $\transf_\xx(g)$ has a single point at infinity; since $g$ is linear, by Proposition \ref{prop:transf-nonsing} the point is also nonsingular for $\overline{\transf_\xx(g)}$ (note that, by construction, it is nonsingular for $\overline{\xx}$). Hence, $\dim \recip(K(t)[\transf_\xx(g)])=1$ and $\dim\recip(K[\xx])=2$, as claimed.
\end{proof}

\begin{ex}\label{ex:quadrics}
Let $K$ be a field of characteristic $\neq 2$, and let $\xx$ be an irreducible quadric, i.e., the zero set in $\aff^3$ of an irreducible polynomial $f(X,Y,Z)$ of degree two. Let $D:=K[\xx]=K[X,Y,Z]/(f)$ be its ring of functions. We claim that $\dim\recip(D)=2$ unless $\xx$ is affinely equivalent to $\Vaff(XY-1)$.

Let $f_2$ be the homogeneous component of degree $2$: we distinguish two cases.

\medskip

Suppose that $f_2$ is irreducible, so that $\xx_\infty$ is irreducible,, and write $\displaystyle{f(X,Y,Z)=\sum_{0\leq i+j+k\leq 2}a_{ijk}X^iY^jZ^k}$. Then, the gradient of the homogenization $F$ of $f$ is
\begin{equation*}
(a_{100}X+a_{010}Y+a_{001}Z+2a_{000}W,2a_{200}X+a_{100}W,2a_{020}Y+a_{010}W,2a_{002}Z+a_{001}W)
\end{equation*}
which at infinity (i.e., when $W=0$) is equal to
\begin{equation*}
(a_{100}X+a_{010}Y+a_{001}Z,2a_{200}X,2a_{020}Y,2a_{002}Z).
\end{equation*}
Since the characteristic of $K$ is not $2$, no point of $\xx_\infty$ is singular. 

If there is no $K$-rational point in $\xx_\infty$, then $\dim\recip(D)=2$ by Proposition \ref{prop:noKrational-deg23}.

Suppose that there is a $K$-rational point $p\in\xx_\infty$. Let $g(X,Y,Z):=\partial_XF(p)X+\partial_YF(p)Y+\partial_ZF(p)Z$, and note that $g\in K[X,Y,Z]$. The tangent line of $\overline{\xx}$ at $p$ is contained in the projective plane $\overline{\Vaff(g)}=\Vproj(g)$ determined by $g$. Thus, $p$ is the only point of intersection between $\overline{\Vaff(g)}$ and $\xx_\infty$, since the multiplicity of the intersection at $p$ is $2$ and $f_2$ has degree $2$. Since $g$ is linear, by Proposition \ref{prop:transf-nonsing}\ref{prop:transf-nonsing:linear} $\overline{\transf_X(g)}$ is nonsingular in $p$; by Proposition \ref{prop:affine-irred-VXFreg} $\dim(\recip(D))=2$.

\medskip

Suppose now that $f_2$ is reducible.

If $f_2$ has a linear factor of multiplicity $2$, then $\xx$ is affinely equivalent to $\Vaff(X^2+d)$ or to $\Vaff(X^2+dY)$ for some $d\in K\setminus\{0\}$. In the former case, $X^2+d$ is irreducible and $D\simeq K(\mu)[Y,Z]$, where $\mu\in\overline{K}$ satisfies $\mu^2=-d$; in particular, $\dim\recip(D)=2$. In the latter case, $D\simeq K[X,Z]$ and so again $\dim\recip(D)=2$.

Suppose that $f_d$ has two distinct linear factors. Then, $\xx$ is affinely equivalent to $\Vaff(XY+tZ+d)$ for some $t,d\in K$. If $t=0$, then $d\neq 0$ (since $\xx$ is irreducible) and so, applying an affine transformation, we can suppose $d=-1$, and
\begin{equation*}
D\simeq\frac{K[X,Y,Z]}{(XY-1)}\simeq K\left[X,X^{-1},Z\right].
\end{equation*}
As $\recip(K[X,X^{-1}])=K(X)$, we have $\recip(D)=\recip(K(X)[Z])=K(X)[Z^{-1}]_{(Z^{-1})}$. In particular, $\dim\recip(D)=1$.

If $t\neq 0$, then $\xx$ is affinely equivalent to $\Vaff(XY+tZ)$, and $D\simeq K[X,Y]$. Hence $\dim\recip(D)=2$.

In particular, we also have that $\dim\recip(D)=1$ if and only if $D$ contains units outside of $K$; it follows that $\dim\recip(D)=\dim(D)-\trdeg(K(\unit(D))/K)$, i.e., the inequality of Proposition \ref{prop:units} is an equality in this case.
\end{ex}

\begin{ex}\label{ex:cubics}
Let $K$ be an algebraically closed field of characteristic $\neq 3$ and let $f(X,Y,Z)$ be a polynomial of degree $3$; let $\xx=\Vaff(f)$ be the cubic surface associated to $f$. Suppose that:
\begin{itemize}
\item all points of $\xx_\infty$ are regular for $\xx$;
\item $\xx_\infty$ is regular (in particular, irreducible).
\end{itemize}
Then, $\dim\recip(K[\xx])=2$.

The curve $\xx_\infty$ is a cubic and, since it is regular, it contains an inflexion point $p$. Let $g\in K[X,Y,Z]$ be a linear polynomial such that $\overline{\Vaff(g)}$ contains the tangent line at $p$: since $p$ is an inflexion point, $\overline{\Vaff(g)}\cap\xx_\infty=\{p\}$. By Proposition \ref{prop:transf-nonsing}\ref{prop:transf-nonsing:notlinear}, $p$ is a nonsingular point of $\overline{\Vaff_X(g)}$, and by Proposition \ref{prop:affine-irred-VXFreg} $\dim\recip(K[\xx])=2$.
\end{ex}

\begin{prop}\label{prop:touchdegn}
Let $K$ be a field of characteristic $p$, and let $f\in K[X,Y,Z]$ be an irreducible polynomial of degree $d$, not a multiple of $p$. Let $f_d$ be its homogeneous component of maximal degree, and suppose that there is a polynomial $g(X,Y,Z)$ such that:
\begin{itemize}
\item $f_d$ is irreducible;
\item $f_d(X,Y,Z)=X^d+g(X,Y,Z)$;
\item $g(X,Y,Z)\neq 0$ and has a linear factor $g_1\neq X$;
\item if $g_1$ is such a factor, $\Vproj(g_1)\cap \Vproj(g/g_1)\cap \Vproj(f_d)=\emptyset$.
\end{itemize}
Then, $\dim\recip(K[X,Y,Z]/(f))=2$. 
\end{prop}
\begin{proof}
Since $f$ and $f_d$ are irreducible, so are $\xx:=\Vaff(f)$ and $\xx_\infty$.

%Since $f$ is irreducible, so io $\xx:=\Vaff(f)$.

Up to a linear transformation, we can suppose without loss of generality that $g_1=Y$; we set $g_2:=g_1/Y$. Every point $[0:x:y:z]\in \Vproj(g_1)\cap \xx_\infty$ must satisfy $y=0$ and, since $Y$ divides $g$, also $x=0$; hence it must be $p:=[0:0:0:1]$.

By Proposition \ref{prop:transf-nonsing}\ref{prop:transf-nonsing:linear}, $p$ is regular in $\overline{\transf_\xx(g)}$; we claim that it is regular in $\overline{\xx}$. Indeed, consider the partial derivative with respect to $Y$ of $F$, the homogenization of $f$. Since $p\in\xx_\infty$, we can ignore the terms of degree $<d$ in $X,Y,Z$, and so
\begin{equation*}
\partial_Y F(W,X,Y,Z)(p)=\partial_Y f_d(W,X,Y,Z)(p)=g_2(p)+Y\partial_Yg_1(p).
\end{equation*}
We have $g_1(p)=0$ while $g_2(p)\neq 0$ (otherwise $p\in \Vproj(g_1)\cap \Vproj(g_2)\cap \Vproj(f_d)$, against the hypothesis). Thus the gradient of $F$ is nonzero in $p$ and $p$ is regular for $\overline{\xx}$.

Thus $\xx$ and $g$ satisfy the hypothesis of Proposition \ref{prop:affine-irred-VXFreg}, and $\recip(K[X,Y,Z]/(f))$ has dimension $2$.
\end{proof}

\begin{cor}\label{cor:touchdegn}
Let $K$ be a field of characteristic $p$, and let $f\in K[X,Y,Z]$ be an irreducible polynomial of degree $d$, not a multiple of $p$. Let $f_d$ be its homogeneous component of maximal degree, and suppose that:
\begin{itemize}
\item $f_d$ is irreducible;
\item $f_d(X,Y,Z)=X^d+g(Y,Z)$;
\item $g(Y,Z)$ has a simple linear factor.
\end{itemize}
Then, $\dim\recip(K[X,Y,Z]/(f))=2$. 
\end{cor}
\begin{proof}
Apply Proposition \ref{prop:touchdegn}.
\end{proof}

The previous proposition allows an almost immediate calculation of the dimension of the reciprocal complement in several cases.
\begin{ex}
Let $f(X,Y,Z)=X^d-YZ^{d-1}+h(X,Y,Z)$, where $h$ has degree $<d$. If $d$ is not a multiple of the characteristic of $K$ and $f$ is irreducible, then $\dim\recip(K[X,Y,Z]/(f))=2$.
\end{ex}

\begin{ex}\label{ex:fermat}
Let $K$ be a field of characteristic $p\neq 2$, and let $m\inN^+$ be an integer not divisible by $p$. Let $f(X,Y,Z):=X^m+Y^m+Z^m+1$. If either
\begin{itemize}
\item $m$ is odd; or
\item $K$ contains a primitive $2m$-root of unity, 
\end{itemize}
then $\dim\recip(K[X,Y,Z]/(f(X,Y,Z)))=2$.

Using the notation of Proposition \ref{prop:touchdegn}, $f_m(X,Y,Z)=X^m+Y^m+Z^m=X^m+g(Y,Z)$ with $g(Y,Z)=Y^m+Z^m$. If $m$ is odd, then $Y+Z$ is a factor of $g$; if $K$ contains a primitive $2m$-root of unity, say $\zeta$, then $X-\zeta Y$ is a linear factor; in both cases, the linear factor is simple since $Y^m+Z^m$ is separable due to the hypothesis $p\not|m$. By Proposition \ref{prop:touchdegn},  $\dim\recip(K[X,Y,Z]/(f(X,Y,Z)))=2$.
\end{ex}

The polynomials in Proposition \ref{prop:touchdegn} and Corollary \ref{cor:touchdegn} use $K$-rational points to prove that the dimension of the reciprocal complement is $2$. Using a method similar to the proof of Proposition \ref{prop:noKrational-deg23}, we can prove a similar statement when such points do not exists.
\begin{prop}\label{prop:noKrational-Xd}
Let $K$ be a field of characteristic $p$, and let $f\in K[X,Y,Z]$ be an irreducible polynomial of prime degree $d$, non divisible by $p$. Let $f_d$ be its homogeneous component of maximal degree, and suppose that there is a polynomial $g(Y,Z)$ such that:
\begin{itemize}
\item $f_d$ is irreducible;
\item $f_d(X,Y,Z)=X^d+g(Y,Z)$;
\item $g(Y,Z)$ is irreducible.
\end{itemize}
Then, $\dim\recip(K[X,Y,Z]/(f))=2$. 
\end{prop}
\begin{proof}
Let $\xx:=\Vaff(f)$. The polynomial $g$ is homogeneous in two variables, and thus it factors linearly in $\overline{K}[Y,Z]$; moreover, since $g$ is irreducible over $K$, its factors $Y-\omega_i Z$ over $\overline{K}$ are conjugated over $K$. In particular, all the coordinates $[0:0:\omega_i:1]$ are conjugated over $K$ and thus they represent the same point in $\proj^3_K$.

Let $h(X,Y,Z):=X$: by the previous reasoning, the system $h=f_d=0$ has a single solution in $\proj^2_K$, i.e., $\transf_\xx(h)$ has a single point $p$ at infinity. Since $p$ does not divide $d$, moreover, $g$ is separable over $K$ and thus $p$ is regular for $\overline{\xx}$. Since $h$ is linear, $p$ is regular for $\overline{\transf_\xx(h)}$ (Proposition \ref{prop:transf-nonsing}) and so $\dim\recip(K(t)[\transf_\xx(h)])=1$. Since $f_d$ is irreducible, it follows that $\dim\recip(K[\xx])=\dim\recip(K[X,Y,Z]/(f))=2$.
\end{proof}

\begin{ex}
Let $K$ be a field of characteristic $p$, and let $q\neq p$ be a prime number. Let $a,t\in K\setminus\{0\}$, and let
\begin{equation*}
f(X,Y,Z):=X^q+tY^q+atZ^q+f_1(X,Y,Z),
\end{equation*}
where the degree of $f_1$ is $<q$. If $f$ is irreducible, then $\dim\recip(K[X,Y,Z]/(f))=2$.

\emph{Fact:} the polynomial $h(T):=T^q-a$ is either irreducible or it has a solution in $K$.

\emph{Proof of fact:} let $\alpha\in\overline{K}$ be a solution of $h$, and let $1\neq\zeta\in\overline{K}$ be a $q$-th root of unity. Note that $\zeta$ exists since $q$ is not equal to the characteristic of $K$. Then, the solutions of $h$ in $\overline{K}$ are $\alpha,\zeta\alpha,\zeta^2\alpha,\ldots,\zeta^{q-1}\alpha$. Suppose that $h$ is not irreducible and let $h_1$ be a proper factor of degree $k$: then, $h_1(0)=\epsilon\alpha^k\zeta^t$ for some $\epsilon\in\{1,-1\}$ and some $t\inN$. As $0<k<q$, there is a $k^\ast$ such that $kk^\ast=rq+1$ for some $r\inN$. Then,
\begin{equation*}
h_1(0)^{qk^\ast}=(\epsilon\alpha^k\zeta^t)^{qk^\ast}=\epsilon^{qk^\ast}(\alpha^q)^{rq+1}(\zeta^q)^{tk^\ast}=\epsilon^{qk^\ast}a^{rq}a.
\end{equation*}
It follows that $\alpha':=h_1(0)^{k^\ast}/(\epsilon^{k^\ast}a^r)$ is $q$-th root of $a$ and thus it is a solution of $h$. Since $h_1(0)\in K$, we have $\alpha'\in K$, i.e., $h$ has a solution in $K$.

\medskip

Suppose first that $h$ has a solution $\alpha\in K$. Then, $f$ satisfies the hypothesis of Corollary \ref{cor:touchdegn} with $g=Y-\alpha Z$ and so $\dim\recip(K[X,Y,Z]/(f))=2$. On the other hand, if $h$ does not have a solution, then by the fact $Y^q+aZ^q$ is irreducible over $K$; by Proposition \ref{prop:noKrational-Xd}, $\dim\recip(K[X,Y,Z]/(f))=2$.
\end{ex}

\begin{comment} problema: come faccio a far vedere X\notin\recip?

We end with an example where $\xx_\infty$ is reducible.
\begin{ex}
Let $K$ be a field of characteristic $\neq 2$. $f(X,Y,Z):=XYZ+X^2+Y^2+Z^2-1$, and let $\xx:=\Vaff(f)$. Note that $f$ is irreducible since it is irreducible as a polynomial in $K[X,Y][Z]$.

Let $g(X,Y,Z):=X$; then, the $g$-transform $\transf_\xx(g)$ is
\begin{equation*}
\transf_\xx(g):=\Vaff(XYZ+X^2+Y^2+Z^2-1,X-t)=\Vaff(Y^2+tYZ+Z^2+t^2-1,X-t).
\end{equation*}
The curve $\Vaff(Y^2+tYZ+Z^2+t^2-1)\subset\aff^2_{K(t)}$ has a single point at infinity, since $Y^2+tYZ+Z^2$ is irreducible over $K(t)$ (the discriminant $t^2-4$ is not a square in $K(t)$). It follows that $\transf_\xx(g)$ has a unique point at infinity; since $g$ is regular, 
\end{ex}
\end{comment}

\section{The projective plane}\label{sect:piano}
In the previous sections, we focused on studying whether the localization $\recip(K[\xx])_{\pp_f}=\recip(K(t)[\transf_\xx(f)])$ is a field. When $\xx=\proj^2$, we can actually be more precise, and understand when it is a DVR.
\begin{prop}\label{prop:curve}
Let $K$ be an infinite field, and let $f\in K[X,Y]$ be a polynomial whose degree is not a multiple of the characteristic of $K$. Let $t,Z$ be independent indeterminates over $K$. Then,
\begin{equation*}
\frac{K[X,Y]}{(f(X,Y))}\simeq K[Z]\iff\frac{K(t)[X,Y]}{(f(X,Y)-t)}\simeq K(t)[Z].
\end{equation*}
\end{prop}
\begin{proof}
By \cite[Corollary 5.3.13]{vanderEssen}, for a field $L$ and a polynomial $f$ satisfying the hypothesis of the theorem we have $L[X,Y]/(f)\simeq L[Z]$ if and only if there is a polynomial $g\in L[X,Y]$ such that $L[X,Y]=L[f,g]$. Since $L[f,g]=L[f-\alpha,g]$ for all $\alpha\in L$, in the claim we can replace $\frac{K(t)[X,Y]}{(f(X,Y)-t)}$ with $\frac{K(t)[X,Y]}{(f(X,Y))}$.

Suppose that $\frac{K[X,Y]}{(f(X,Y))}\simeq K[Z]$, so that $K[X,Y]=K[f,g]$ for some $g\in K[X,Y]$. Then, $X=p(f,g)$, $Y=q(f,g)$ for some polynomials $p,q\in K[u,v]\subseteq K(t)[u,v]$; in particular, $K(t)[X,Y]=K(t)[f,g]$ and so $\frac{K(t)[X,Y]}{(f(X,Y)-t)}\simeq K(t)[Z]$.

Conversely, suppose that $\frac{K(t)[X,Y]}{(f(X,Y))}\simeq K(t)[Z]$: then, $K(t)[X,Y]=K(t)[f,g]$ for some $g\in K(t)[X,Y]$. In particular, $X=p(f,g)$ for some $p\in K(t)[u,v]$. Let $p=\sum_{i,j}\lambda_{ij}u^iv^j$, with $\lambda_{ij}\in K(t)$. Since $K$ is infinite, we can find an $\alpha\in K$ where every $\lambda_{ij}$ is defined (i.e., such that $\alpha$ is not a zero of the denominators of any $\lambda_{ij}$). Then, 
\begin{equation*}
X=X(\alpha)=\sum_{i,j}\lambda_{ij}(\alpha)f^ig^j\in K[f,g].
\end{equation*}
By the same reasoning, $Y\in K[f,g]$, and thus $K[f,g]=K[X,Y]$. Therefore $\frac{K[X,Y]}{(f(X,Y))}\simeq K[Z]$.
\end{proof}

\begin{teor}\label{teor:locallizz-recipXY}
Let $K$ be an infinite field and let $f\in K[X,Y]$ be an irreducible polynomial whose degree is not a multiple of the characteristic of $K$. Let $D_f:=K[X,Y]/(f)$ and let $\pp_f$ be the largest prime ideal of $\recip(K[X,Y])$ not containing $f$. Then, the following hold.
\begin{enumerate}[(a)]
\item\label{teor:locallizz-recipXY:field} $\recip(K[X,Y])_{\pp_f}$ is a field if and only $\recip(D_f)$ is a field.
\item\label{teor:locallizz-recipXY:ic} $\recip(K[X,Y])_{\pp_f}$ is integrally closed if and only if $\recip(D_f)$ is integrally closed.
\end{enumerate}
\end{teor}
\begin{proof}
\ref{teor:locallizz-recipXY:field} Let $d$ be the degree of $f$ and $F$ its homogenization in the new variable $W$. By Lemma \ref{lemma:numsol-Kt}, the number of solutions of $F=W=0$ in $K$ and in $K(t)$ are the same; as $d>1$, moreover, the homogenization of $f-t$ is $F-tW^d$, and thus the solutions of $F=W=0$ and $F-tW^d=W=0$ are the same. It follows that the number of points at infinity of $D_f$ is the same as the number of points at infinity of the transform $\transf_{\proj^2}(f)$; hence, $\recip(D_f)$ is a field if and only if $\recip(K(t)[\transf_{\proj^2}(f)])$ is a field. By Proposition \ref{prop:transf}\ref{prop:transf:recip},  $\recip(K[X,Y])_{\pp_f}\simeq\recip(K(t)[\transf_{\proj^2}(f)])$; the claim is proved.

\ref{teor:locallizz-recipXY:ic} By the previous point, we can suppose without loss of generality that $\recip(D_f)$ and $\recip(K[X,Y])_{\pp_f}$ are not fields. Then, they are both one-dimensional local domains.

Using \cite[Proposition 2.3]{reciprocal-curves} (or \cite[Theorem 6.1]{guerrieri-recip}), Proposition \ref{prop:curve} and Proposition  \ref{prop:transf}\ref{prop:transf:recip}, we have
\begin{equation*}
\begin{aligned}
\recip(D_f) \text{~is a DVR}\iff & D_f\simeq K[Z] \text{~for some indeterminate~}Z\\
\iff & \frac{K(t)[X,Y]}{(f(X,Y)-t)}\simeq K(t)[Z]\\
\iff & \recip\left(\frac{K(t)[X,Y]}{(f(X,Y)-t)}\right) \text{~is a DVR}\\
\iff & \recip(K[X,Y])_{\pp_f}\text{~is a DVR}.
\end{aligned}
\end{equation*}
The claim is proved.
\end{proof}

The following corollary has been proved for arbitrary fields in \cite[Theorem 5.8]{EGL-RD-polinomyal} by explicit computation.
\begin{cor}
Let $K$ be an infinite field. then, $\recip(K[X,Y])$ is not integrally closed.
\end{cor}
\begin{proof}
Let $p$ be the characteristic of $K$.

We prove that some localization of $\recip(K[X,Y])$ is not integrally closed; by Theorem \ref{teor:locallizz-recipXY}, it is enough to find a polynomial $f\in K[X,Y]$ such that:
\begin{itemize}
\item the degree of $f$ is not a multiple of $p$;
\item $\xx:=\Vaff(f)$ is regular and has a single point at infinity;
\item $\xx$ is not rational.
\end{itemize}
If $p\neq 2$, take $d$ such that $p$ does not divide $d$ nor $d-1$ and consider $f(X,Y):=X^d+Y^{d-1}+1$: then, $D_f$ has a single point at infinity (namely, $[0:0:1]$), is regular and has genus $(d-1)(d-2)/2>0$ \cite[0BYD]{stacks-project}, so it is not rational.

If $p=2$, the curve defined by $f(X,Y):=X^3+Y^2+Y^2+XY+1$ has the same properties. In all cases, $\recip(K[X,Y])$ is not integrally closed.
\end{proof}

\bibliographystyle{plain}
\bibliography{/bib/miei,/bib/articoli,/bib/libri,/bib/eventualia}
\end{document}